\documentclass[11pt]{amsart}
\usepackage{appendix}
\usepackage{graphicx}
\usepackage{chngpage}
\usepackage{multirow}
\usepackage{hyperref}
\usepackage[all,cmtip]{xy}
\usepackage{amssymb}
\usepackage{amsmath}
\usepackage{longtable}

\usepackage{pstricks}
\usepackage{color}
\theoremstyle{plain}
\newtheorem{thm}{Theorem}[section]
\newtheorem{theorem}[thm]{Theorem}

\newtheorem{lemma}[thm]{Lemma}

\newtheorem{proposition}[thm]{Proposition}
\theoremstyle{definition}
\newtheorem{remark}[thm]{Remark}

\newtheorem{definition}[thm]{Definition}

\newtheorem{example}[thm]{Example}

\numberwithin{equation}{section}
\newcommand{\Aut}{{\rm Aut}}
\newcommand{\Ker}{{\rm Ker}}

\newcommand{\Diag}{{\rm diag}}

\newcommand{\SmallGroup}{{\rm SmallGroup}}

\newcommand{\GL}{{\rm GL}}
\newcommand{\PGL}{{\rm PGL}}
\newcommand{\PSL}{{\rm PSL}}
\newcommand{\Det}{{\rm det}}
\newcommand{\SL}{{\rm SL}}

\newcommand{\C}{{\mathbb C}}

\renewcommand{\P}{{\mathbb P}}

\newcommand{\Z}{{\mathbb Z}}

 \title{Automorphism groups of smooth cubic threefolds}

 \author{Li Wei and Xun Yu}

\address{Center for Applied Mathematics, Tianjin University, Weijin Road 92, Tianjin 300072, China}
\email{WeiLitju@163.com, xunyu@tju.edu.cn}

\begin{document}

\begin{abstract} In this paper, we classify groups which faithfully act on smooth cubic threefolds. It turns out that there are exactly $6$ maximal ones and  we describe them with explicit examples of target cubic threefolds.

\end{abstract}

\maketitle

\section{Introduction}\label{Intro}

Throughout this paper, we work over the complex number field $\mathbb{C}$. The purpose of this paper is to study the automorphism group of a smooth cubic threefold, an important counterexample to the three-dimensional  L\"{u}roth problem (\cite{CG72}).  Let $X\subset \P^4$ be a smooth cubic threefold defined by an irreducible homogeneous polynomial $F$.  It is known (\cite{MM63}) that the automorphism group $\Aut(X)$ is finite, every element of $\Aut(X)$ extends to an automorphism of the ambient space $\P^4$, and two smooth cubic threefolds are isomorphic if and only if they are projective linearly isomorphic. Classification of finite groups appearing as subgroups of $\Aut(X)$ seems still unkown (for studies of automorphism groups of smooth cubic hypersurfaces, see \cite{Se42}, \cite{Ad78}, \cite{Ho97}, \cite{Ro09},\cite{GL11}, \cite{Do12}, \cite{GL13}, \cite{Mo13}, \cite{HM14}, \cite{BCS16}, \cite{Fu16},\cite{DD18}, \cite{LZ19}, etc). Our main result is the following (see Section \ref{ss:6examples} for more details about the $6$ groups in the theorem)

\begin{theorem}[Theorem \ref{thm:Main}]
A finite group $G$ has a faithful action on a smooth cubic threefold if and only if $G$ is isomorphic to a subgroup of one of the following $6$ groups: $C_3^4\rtimes S_5$, $((C_3^2\rtimes C_3)\rtimes C_4)\times S_3$, $C_{24}$, $C_{16}$, $\PSL(2,11)$, $C_3\times S_5$.
\end{theorem}

The automorphism group of Klein cubic threefold is isomorphic to the finite simple group $\PSL(2,11)$ (\cite{Ad78}), and it is, up to isomorphism, the unique smooth cubic threefold admitting an automorphism of order $11$ (\cite{Ro09}). On the other hand, by a result of \cite{GL11}, all possible prime orders of automorphisms of smooth cubic threefolds are $2,3,5,11$. Therefore, in order to classify subgroups of $\Aut(X)$,  we are reduced to consider subgroups of order $2^a3^b5^c$. Our approch to classify such subgroups is the same as that of \cite{OY19} in which all possible groups of automorphisms of smooth quintic threefolds are classified.  By Matsumura-Monsky (\cite{MM63}), it suffices to consider finite subgroups $G\subset \PGL(5,\C)=\Aut(\P^4)$ such that $F$ is $G$-invariant (i.e., for each $g=[A] \in G$, $A(F)=\lambda_A  F$ for some $\lambda_A\in \C^*$, where $A\in \GL(5,\C)$ is a representative of $g$).  As in \cite{OY19}, we use the notion $F$-liftability (see Definition \ref{def:Flift} in Section \ref{ss:smlift}) to transfer classification problem in $\PGL(5,\C)$ to classification problem in $\GL(5,\C)$, and the latter one can be handled with the help of non-smoothness criteria (Proposition \ref{pp:nsmcubic}) and computer program GAP (see Appendix \ref{ap:GAP}). It turns out that for a smooth cubic threefold $X$, $\Aut(X)\subset \PGL(5,\C)$ always admits an $F$-lifting (Theorem \ref{thm:Flift}). For a candidate finite group $G$ of order $2^a3^b 5^c\le 2000$,  smoothness of $X$ gives strong constraints (see Table \ref{tab:ab}) on eigenvalues (``local information'') of each element in an $F$-lfitng $\tilde{G}$ (which is naturally viewed as a 5-dimensional faithful linear representation of $G$) . On the other hand, GAP provides character table and list of all subgroups (especially, abelian subgroups) of $G$ (``global information''). Then our way of ruling out groups is simply by combining local and global information (see Theorem \ref{thm:sylow2} and its proof). It turns out that this method of ruling out groups which cannot faithfully act on smooth cubic threefolds is quite efficient in our study. Larger orders cases (i.e., $|G|=2^a3^b 5^c>2000$) are reduced to smaller orders cases just mentioned in PC free way.

\medskip

\subsection*{Acknowledgements}
We would like to thank Professors Keiji Oguiso and Song Yang for helpful conversations. The second author is partially supported by the National Natural Science Foundation of China (Grant No. 11701413).

\section{Notation and conventions}

In this paper, if $A\in \GL(n,\mathbb{C})$, then we use $[A]$ denote the corresponding element in $\PGL(n, \mathbb{C})$.

   $X_i$ (resp. $G_{X_i}$), $i=1,...,6$, are the six smooth cubic threefolds (resp. finite groups) in Example \ref{mainex} in Section \ref{ss:6examples};

  $I_n:=$ the identity matrix of rank $n$;

  $\xi_k:=e^{\frac{2\pi i}{k}}$ a $k$-th primitive root of unity, where $k$ is a positive integer;

     We use $\pi: \GL(n,\mathbb{C})\longrightarrow \PGL(n,\mathbb{C})$ to denote the natural quotient map.

     Let $G$ be a finite group and $p$ be a prime. If no confusion causes, we use $G_p$ to denote a Sylow $p$-subgroup of $G$.

    The following is the list of symbols of finite groups used in this article:

     $C_n$: a cyclic group of order $n$,

     $D_{2n}$: a dihedral group of order $2n$,

     $S_n(A_n)$: a symmetric (alternative) group of degree $n$,

     $Q_8$: a quaternion group of order 8.

\section{Examples and main Theorem}\label{ss:6examples}

Throughout this paper, we identity $\PGL(n+1,\C)$ with $\Aut(\P^n)$ via the following group action: \begin{equation}\label{eq:Psi} \Psi:\PGL(n+1,\C)\times \P^n\longrightarrow \P^n,
\end{equation} where, for any $A=(a_{ij})\in \GL(n+1,\C)$ and any $(z_1:\cdots :z_{n+1})\in \P^n$, $$\Psi([A],(z_1:\cdots :z_{n+1}))=(\sum_{i=1}^{n+1}a_{1i}z_i:\cdots:\sum_{i=1}^{n+1}a_{(n+1)i}z_i).$$

For any $A=(a_{ij})\in \GL(n+1,\mathbb{C})$ and any homogeneous polynomial $F\in \C[x_1,...,x_{n+1}]$, we denote by $A(F)$ the homogeneous polynomial 

\begin{equation}\label{eq:A(F)} F(\sum_{i=1}^{n+1}a_{1i}x_i,\cdots,\sum_{i=1}^{n+1}a_{(n+1)i}x_i)\, .\end{equation}
For a finite subgroup $G<\PGL(n+1,\C)$ and a smooth hypersurface $X\subset \P^n$ defined by an irreducible homogeneous polynomial $F=F(x_1,...,x_{n+1})$ of degree greater than $1$, if, for any $[A]\in G$, $A(F)=\lambda_{A} F$ for some $\lambda_{A}\in \C^*$, then clearly $G$ acts on $X$ via $\Psi$ and $G$ is a subgroup of $\Aut(X)$.

\begin{example}\label{mainex}
(1) Fermat cubic threefold $X_1$: $F=x^3_1+x^3_2+x^3_3+x^3_4+x^3_5=0$. Let $G_{X_1}$ be the subgroup of $\PGL(5,\mathbb{C})$ generated by the following three matrices:

$A_{11}=
    \begin{pmatrix}
   0&1&0&0&0\\
   1&0&0&0&0\\
   0&0&1&0&0\\
   0&0&0&1&0\\
   0&0&0&0&1
    \end{pmatrix} A_{12}=
    \begin{pmatrix}
    0&1&0&0&0\\
    0&0&1&0&0\\
    0&0&0&1&0\\
    0&0&0&0&1\\
    1&0&0&0&0
    \end{pmatrix}$ $A_{13}=
    \begin{pmatrix}
    1&0&0&0&0\\
    0&\xi_3&0&0&0\\
    0&0&\xi_3^2&0&0\\
    0&0&0&1&0\\
    0&0&0&0&1
    \end{pmatrix}.$

Then $\Aut(X_1)=G_{X_1}\cong C_3^4\rtimes S_5$ and $|G_{X_1}|=2^3\cdot 3^5\cdot 5=9720$.
\\[.2cm]

(2) Let $X_2: F=x^3_1+x^3_2+x^3_3+3(\sqrt{3}-1)x_1x_2x_3+x^3_4+x^3_5=0$ (see \cite[Lemma 12.15]{DD18}) and let $G_{X_2}$ be the subgroup of $\PGL(5,\mathbb{C})$ generated by the following five matrices:

$A_{21}=
    \begin{pmatrix}
    0&1&0&0&0\\
    0&0&1&0&0\\
    1&0&0&0&0\\
    0&0&0&1&0\\
    0&0&0&0&1
    \end{pmatrix} A_{22}=
    \begin{pmatrix}
    1&0&0&0&0\\
    0&\xi_3&0&0&0\\
    0&0&\xi_3^2&0&0\\
    0&0&0&1&0\\
    0&0&0&0&1
    \end{pmatrix} A_{23}=\frac{1}{\sqrt{3}}
    \begin{pmatrix}
    1&1&1&0&0\\
    1&\xi_3&\xi_3^2&0&0\\
    1&\xi_3^2&\xi_3&0&0\\
    0&0&0&\sqrt{3}&0\\
    0&0&0&0&\sqrt{3}
    \end{pmatrix}$
    
    $A_{24}=
    \begin{pmatrix}
    1&0&0&0&0\\
    0&1&0&0&0\\
    0&0&1&0&0\\
    0&0&0&0&1\\
    0&0&0&1&0
    \end{pmatrix} A_{25}=
    \begin{pmatrix}
    1&0&0&0&0\\
    0&1&0&0&0\\
    0&0&1&0&0\\
    0&0&0&\xi_3&0\\
    0&0&0&0&\xi_3^2
    \end{pmatrix}.$

  Then $G_{X_2}$ acts on $X_2$, $G_{X_2}\cong ((C_3^2\rtimes C_3)\rtimes C_4)\times S_3$, and $|G_{X_2}|=2^3\cdot 3^4=648$.\\[.2cm]

 (3) Let $X_3: F=x^2_1x_2+x^2_2x_3+x^2_3x_4+x^3_4+x^3_5=0$, and let $G_{X_3}$ be the subgroup of $\PGL(5,\mathbb{C})$ generated by the following matrice:

 $A_{31}=
    \begin{pmatrix}
    \xi_8&0&0&0&0 \\
    0&\xi_8^{-2}&0&0&0 \\
    0&0&-1&0&0 \\
    0&0&0&1&0 \\
    0&0&0&0&\xi_3 \\
    \end{pmatrix}$

 Then $G_{X_3}$ acts on $X_3$, $G_{X_3}$ is isomorphic to $C_{24}$ and $|G_{X_3}|=2^3\cdot 3=24$. \\[.2cm]

 (4) Let $X_4: F=x^2_1x_2+x^2_2x_3+x^2_3x_4+x^2_4x_5+x^3_5=0$ and let $G_{X_4}$ be the subgroup of $\PGL(5,\mathbb{C})$ generated by the following matrice:

 $A_{41}=
    \begin{pmatrix}
    \xi_{16}&0&0&0&0\\
    0&\xi_{16}^{-2}&0&0&0\\
    0&0&\xi_{16}^{4}&0&0\\
    0&0&0&-1&0\\
    0&0&0&0&1
    \end{pmatrix} $

 Then $G_{X_4}$ acts on $X_4$, $G_{X_4}\cong C_{16}$ and $|G_{X_4}|=2^4=16$.\\[.2cm]

 (5) Klein cubic threefold $X_5:F=x^2_1x_2+x^2_2x_3+x^2_3x_4+x^2_4x_5+x^2_5x_1=0$ and let $G_{X_5}$ be the finite simple group $\PSL(2,11)$. Then $G_{X_5}\cong \Aut(X_5)$ (\cite{Ad78}) and $|G_{X_5}|=2^2\cdot 3\cdot 5\cdot 11=660$. Note that $\PSL(2,11)$ is not a subgroup of the Cremona group of rank 3 (\cite{Pr12}). \\[.2cm]

   (6) Let $X_6: \{x_1^3+x_2^3+x_3^3+x_4^3+x_5^3+x_6^3=x_1+x_2+x_3+x_4+x_5=0\}\subset \mathbb{P}^5$ and let $G_{X_6}$ be the subgroup of $\PGL(6,\C)$ generated by the following three matrices:

   $A_{61}=
    \begin{pmatrix}
    0&1&0&0&0&0 \\
    1&0&0&0&0&0 \\
    0&0&1&0&0&0 \\
    0&0&0&1&0&0 \\
    0&0&0&0&1&0 \\
    0&0&0&0&0&1 \\
    \end{pmatrix} A_{62}=
    \begin{pmatrix}
    0&1&0&0&0&0 \\
    0&0&1&0&0&0 \\
    0&0&0&1&0&0 \\
    0&0&0&0&1&0 \\
    1&0&0&0&0&0 \\
    0&0&0&0&0&1 \\
    \end{pmatrix} A_{63}=
    \begin{pmatrix}
   1&0&0&0&0&0 \\
    0&1&0&0&0&0 \\
    0&0&1&0&0&0 \\
    0&0&0&1&0&0 \\
    0&0&0&0&1&0 \\
    0&0&0&0&0&\xi_3 \\
    \end{pmatrix}$

   Then $G_{X_6}$ acts on $X_6$, $G_{X_6}\cong S_5\times C_3$ and $|G_{X_6}|=2^3\cdot 3^2\cdot 5=360$.   \\[.2cm]

  \end{example}

  Our main theorem is the following:
  \begin{theorem}\label{thm:Main}

  For a finite group $G$, the following two conditions are equivalent to each other:

  (i) $G$ is isomorphic to a subgroup of one of the $6$ groups above, and

  (ii) $G$ has a faithful action on a smooth cubic threefold.
\end{theorem}

We will prove Theorem \ref{thm:Main} in Section \ref{ss:proofmainthm}.

 \section{Smoothness and liftability}\label{ss:smlift}
 
 In this section, we recall some definitions and results from \cite{OY19}, and we will prove that any subgroup of the automorphism group of a smooth cubic threefold has an $F$-lifting (Theorem \ref{thm:Flift}).

\begin{definition}
   Let $F=F(x_1,...,x_{n+1})$ be a homogeneous polynomial of degree $d> 0$ and let $m=m(x_1,...,x_{n+1})$ be a monomial of degree $d$. Then we say $m$ is in $F$ (or $m\in F$) if the coefficient of $m$ is not zero in the expression of $F$. 
\end{definition}

We will frequently use Lemm \ref{lem:nsm} and Proposition \ref{pp:nsmcubic} to check smoothness of hypersurfaces in the sequel.

\begin{lemma}[{\cite[Lemma 3.2 and Proposition 3.3]{OY19}}]\label{lem:nsm}
Let $F=F(x_1,...,x_{n+1})$ be an irreducible homogeneous polynomial of degree $d\geq 3$ and let $M:=\{F=0\}\subseteq \mathbb{P}^n$. Let $a$ and $b$ be two nonnegative integers, and $2a+b\leq n$. The hypersurface $M$ is not smooth if there exist $a+b$ distinct variables $x_{i_1},\ldots,x_{i_{a+b}}$ such that $F\in ( x_{i_1},\ldots,x_{i_a}) +( x_{i_{a+1}},\ldots,x_{i_{a+b}}) ^2$, where $( x_{k_1},\cdots,x_{k_m}) $ means the ideal of $\mathbb{C}[x_1,...,x_{n+1}]$ generated by $x_{k_1},\cdots,x_{k_m}$. In particular, if $M$ is smooth, then, for any $i\in \{1,2,..., n+1\}$,  $x_i^{d-1}x_j\in F$ for some $j=j(i)$.

\end{lemma}

 \begin{proposition}[{\cite[Proposition 3.4]{OY19}}]\label{pp:nsmcubic}
Let $M$ be a hypersurface in $\mathbb{P}^4$ defined by an irreducible homogeneous polynomial $F=F(x_1,...,x_5)$ of degree 3. Then $M$ is not smooth if one of the following three conditions is true:

(1) There exists $ 1\leq i\leq 5$, such that for all $ 1\leq j\leq 5,  x_i^2x_j\notin F$;

(2) There exists a pair $ (p, q), p\neq q,$  such that  $F\in ( x_p,x_q) $;

(3) There exist three distinct variables $x_i,x_j,x_k$, such that $F\in ( x_i) +( x_j,x_k) ^2$.
\end{proposition}

  Let $F\in \C[x_1,...,x_{n+1}]$ be a homogeneous polynomial of degree $d$. For any $A=(a_{ij})\in \GL(n+1,\mathbb{C})$, we denote by $A(F)$ the homogeneous polynomial $$F(\sum_{i=1}^{n+1}a_{1i}x_i,\cdots,\sum_{i=1}^{n+1}a_{(n+1)i}x_i)\, .$$ Note that $(AB)(F)=A(B(F))$ for any $A,B\in \GL(n+1,\C). $ Following \cite{OY19}, we recall some definitions about liftability of group actions.
  
   \begin{definition}\label{def:leaveinvariant}
(1)  Let $A=(a_{ij})\in \GL(n+1,\mathbb{C})$. We say $F$ is $A$-{\it invariant} if $A(F)=F$. In this case, we also say $A$ leaves $F$ invariant, or $F$ is invariant by $A$. We say $F$ is $A$-{\it semi-invariant} if $A(F)=\lambda F,$ for some $\lambda\in \mathbb{C}^*$.

  (2) Let $G$ be a finite subgroup of $\PGL(n+1,\mathbb{C})$. We say $F$ is $G$-{\it invariant} if for all $g\in G$, there exists $ A_g\in \GL(n+1,\mathbb{C})$ such that $g=[A_g]$ and $A_g(F)=F$, equivalently $F$ is $G$-invariant if $F$ is $A$-semi-invariant for any $A\in \GL(n+1,\C)$ such that $[A]\in \PGL(n+1,\C)$.      
 
 (3) Let $G$ be a finite subgroup of $\PGL(n+1,\mathbb{C})$. We say a subgroup $\widetilde{G}< \GL(n+1,\mathbb{C})$ is a {\it lifting} of $G$ if $\widetilde{G}$ and $G$ are isomorphic via the natural projection $\pi: \GL(n+1,\C)\rightarrow \PGL(n+1,\C).$ We call $G$ {\it liftable} if $G$ admits a lifting.
 \end{definition}

  \begin{definition}\label{def:Flift}
 (1) Let $G$ be a finite subgroup of $\PGL(n+1,\mathbb{C})$. We say $G$ is $F$-{\it liftable} if the following two conditions are satisfied:
  
  1) $G$ admits a lifting $\widetilde{G}< \GL(n+1,\mathbb{C})$; and
  
  2)  $A(F)=F$, for all $A$ in $\widetilde{G}$.

  In this case, we say $\widetilde{G}$ is an $F$-{\it lifting} of $G$.
  
  We say $G$ is $F$-{\it semi-liftable} if 2) is replaced by the following:
  
  2$)^\prime$ for all $A$ in $\widetilde{G}$, $A(F)=\lambda_A F$, for some $\lambda_A\in \mathbb{C}^*$ (depending on $A$).

(2)  Let $h$ be an element in $\PGL(n+1,\mathbb{C})$ of finite order. As a special case, we say $H\in \GL(n+1,\mathbb{C})$ is an $F$-{\it lifting} (resp. a lifting) of $h$ if $\pi (H)=h$ and the subgroup $\langle H\rangle< \GL(n+1,\C) $ is an $F$-lifting (resp. a lifting) of the subgroup $\langle h\rangle<\PGL(n+1,\C) .$ 
  \end{definition}

 \begin{theorem}[{\cite[Theorem 4.8]{OY19}}]\label{liftable}
  Let $G$ be a finite subgroup of $\PGL(n+1,\mathbb{C})$. Let $F\in \C[x_1,...,x_{n+1}]$ be a nonzero homogeneous polynomial of degree $p$, where $p$ is a prime number. Suppose $F$ is $G$-invariant. Let $G_p$ be a Sylow $p$-subgroup. Then $G$ is $F$-liftable if the following two conditions are satisfied:
  
    (1) $G_p$ is $F$-liftable; and

    (2) either $G_p$ has no element of order $p^2$ or $G$ has no normal subgroup of index $p.$  
  \end{theorem}

In the rest of this section, let $X\subset \P^4$ be a smooth cubic threefold defined by an irreducible homogeneous polynomial $F=F(x_1,...,x_5)$. Since $\Aut(X)=\{f\in \Aut(\P^4) | f(X)=X\}$ is a finite group, we may and will view $\Aut(X)$ as a subgroup of $\PGL(5,\C)$ via $\Psi$ (see (\ref{eq:Psi}) in Section \ref{ss:6examples}). 

\begin{proposition}\label{pp:not3lift}
Let $G\subset \Aut(X)$ be a subgroup. Suppose $3\nmid |G|$. Then $G$ has a unique $F$-lifting.
\end{proposition}

\begin{proof}
Let $k$ be $|G|$.  Since $3\nmid k$, by Theorem \ref{liftable}, $G$ admits an $F$-lifting $\tilde{G}\subset \GL(5,\mathbb{C})$. Let $g\in G$.
Suppose $A,B\in \GL(5,\C)$ are two $F$-liftings of $g$. Then $A=\lambda B$. By $A(F)=B(F)=F$,we have
$\lambda^3=1$. By $\lambda I_5=A{B}^{-1}$ and  ${\rm ord}(A)={\rm ord}(B)={\rm ord}(g)$, we have $\lambda^{{\rm ord}(g)}=1$. Then $\lambda^k=1$. Since $3$ and $k$ are coprime, it follows that $\lambda=1$. Thus, $G$ has a unique $F$-lifting $\tilde{G}$. \end{proof}

\begin{lemma}\label{lem:3lift}
Let $g\in \Aut(X)$ of order $3$. If $A$ is a lifting of $g$, then $A$ is an $F$-lifting of $g$. In particular, $g$ admits an $F$-lifting.
\end{lemma}

\begin{proof}
Note that any finite order element of $\PGL(5,\C)$ admits a lifting. Let $A$ be a lifting of $g$. Then $g=[A]$, ${\rm ord}(A)={\rm ord}(g)=3$, and $A(F)=\lambda F$ for some $\lambda\in \C^*$. Our goal is to show $\lambda=1$. Since $F=A^3(F)=\lambda^3F$, it follows that $\lambda^3=1$. For any $i\in \{0,1,2\}$ and any $j\in \{1,2\}$, $\xi_3^iA^j$ is a lifting of $g^j$, and $(\xi_3^iA^j)(F)=(\xi_3^i)^3\lambda^j F=\lambda^j F$. Thus, $\xi_3^iA^j$ is an $F$-lifting of $g^j$ if and only if $A$ is an $F$-lifting of $g$. Then, by linear change of coordinates and replacing the pair $(A, g)$ by $(\xi_3^iA^j, g^j)$ for suitable $i\in \{0,1,2\}$ and $j\in \{1,2\}$, we may assume $A$ is one of the following four cases: (a) $\Diag(1,1,1,1,\xi_3)$, (b) $\Diag(1,1,1,\xi_3,\xi_3^2)$, (c) $\Diag(1,1,1,\xi_3,\xi_3)$, (d) $\Diag(1,1,\xi_3,\xi_3,\xi_3^2)$. 

Case (a): $A=\Diag(1,1,1,1,\xi_3)$. Since $F$ is irreducible, it follows that $F\notin (x_5)\subset \C[x_1,...,x_5]$ and there exists a monomial $x_1^{d_1}x_2^{d_2}x_3^{d_3}x_4^{d_4}\in F$, where all $d_i\ge 0$ and $d_1+d_2+d_3+d_4=3$. Then, by $A(F)=\lambda F$ and $A(x_1^{d_1}x_2^{d_2}x_3^{d_3}x_4^{d_4})=x_1^{d_1}x_2^{d_2}x_3^{d_3}x_4^{d_4}$, we have $\lambda=1$.

Case (b): $A=\Diag(1,1,1,\xi_3,\xi_3^2)$. Since $X$ is smooth, by Proposition \ref{pp:nsmcubic} (2), $F\notin (x_4, x_5)\subset \C[x_1,...,x_5]$ and there exists a monomial $x_1^{d_1}x_2^{d_2}x_3^{d_3}\in F$, where all $d_i\ge 0$ and $d_1+d_2+d_3=3$. Then, by $A(F)=\lambda F$ and $A(x_1^{d_1}x_2^{d_2}x_3^{d_3})=x_1^{d_1}x_2^{d_2}x_3^{d_3}$, we have $\lambda=1$.

Case (c): $A=\Diag(1,1,1,\xi_3,\xi_3)$. Similar to case (b), we have $\lambda=1$.

Case (d): $A=\Diag(1,1,\xi_3,\xi_3,\xi_3^2)$. Suppose $\lambda=\xi_3$. Let $x_1^{d_1}x_2^{d_2}x_3^{d_3}x_4^{d_4}x_5^{d_5}\in F$ where all $d_i\ge 0$ and $d_1+...+d_5=3$. By $A(F)=\lambda F=\xi_3 F$ and $A(x_1^{d_1}x_2^{d_2}x_3^{d_3}x_4^{d_4}x_5^{d_5})=\xi_3^{d_3+d_4+2d_5}x_1^{d_1}x_2^{d_2}x_3^{d_3}x_4^{d_4}x_5^{d_5}$, we have $\xi_3=\xi_3^{d_3+d_4+2d_5}$. Thus, either $d_5>0$, or $d_5=0$ and $d_1+d_2=2$. Then  $x_1^{d_1}x_2^{d_2}x_3^{d_3}x_4^{d_4}x_5^{d_5}\in (x_5)+(x_1,x_2)^2\subset \C[x_1,...,x_5]$. Thus, $F\in (x_5)+(x_1,x_2)^2$, by Proposition \ref{pp:nsmcubic} (3), a contradition to smoothness of $X$. So $\lambda=\xi_3$ is impossible. Similarly, $\lambda=\xi_3^2$ is also impossible (otherwise, $F\in (x_5)+(x_3,x_4)^2$, a contradiction). Thus, we conclude $\lambda=1$. \end{proof}

\begin{lemma}\label{lem:3plift}
Let $g\in \Aut(X)$ of order $3^k$, $k\ge 1$.Then $g$ admits an $F$-lifting.
\end{lemma}

\begin{proof}
If $k=1$, then, by Lemma \ref{lem:3lift}, $g$ admits an $F$-lifting. 

From now on, we may assume $k\ge 2$. Let $A$ be a lifting of $g$. Then ${\rm ord}(A)={\rm ord}(g)=3^k$, and $A(F)=\lambda F$ for some $\lambda\in\C^*$.  Note that $A^{3^{k-1}}$ is a lifting of $g^{3^{k-1}}$. Since ${\rm ord}(g^{3^{k-1}})=3$, by Lemma \ref{lem:3lift}, $A^{3^{k-1}}$ is an $F$-lifting of $g^{3^{k-1}}$. Then $F=A^{3^{k-1}}(F)=\lambda^{3^{k-1}}F$. Thus, $\lambda^{3^{k-1}}=1$. Choose any $\alpha\in \C^*$ such that $\alpha^3=\lambda^{-1}$. Then $\alpha A$ is an $F$-lifting of $g$. \end{proof}

\begin{proposition}\label{pp:3plift}
Let $g\in \Aut(X)$ of order $3^k$, $k\ge 1$.Then $k\le 2$, and  $g$ admits a unique $F$-lifting in $ \SL(5,\C)$. 

\end{proposition}

\begin{proof}
First we show $k\le 2$, and it suffices to show that there exists no $g\in \Aut(X)$ of order 27. Suppose $g\in \Aut(X)$ is of order $27$. By Lemma \ref{lem:3plift}, $g$ admits an $F$-lifting, say $A$. Since $A$ is of finite order, up to linear change of coordinates, we may assume $A$ is diagonal. Replacing $A$ by suitable power $A^j$ where $j$ and $3$ are coprime, we may assume $A=\Diag(\xi_{27},\xi_{27}^{a_1},\xi_{27}^{a_2},\xi_{27}^{a_3},\xi_{27}^{a_4})$. By $A(F)=F$ and $A(x_1^3)=\xi_{27}^3x_1^3\neq x_1^3$, we have $x_1^3\notin F$. Then, by Proposition \ref{pp:nsmcubic} (1), $x_1^2x_j\in F$ for some $j\in \{2,3,4,5\}$. Thus, by suitable linear change of coordinates, we may assume $j=2$ and $x_1^2x_2\in F$. Then by $A(F)=F$ and $x_1^2x_2=A(x_1^2x_2)=\xi_{27}^{2+a_1}x_1^2x_2$, we have $A=\Diag(\xi_{27},\xi_{27}^{-2},\xi_{27}^{a_2},\xi_{27}^{a_3},\xi_{27}^{a_4})$. Repeating the process above, we can show that, up to linear change of coordinates, we may assume $x_1^2x_2$, $x_2^2x_3$, $x_3^2x_4$, $x_4^2x_5\in F$, and $A=\Diag(\xi_{27},\xi_{27}^{-2},\xi_{27}^{4},\xi_{27}^{-8},\xi_{27}^{16})$. Then,  for any $i\in \{1,2,3,4,5\}$, $A(x_5^2x_i)\neq x_5^2x_i$, and, by $A(F)=F$, $x_5^2x_i\notin F$, contradicting to smoothness of $X$ (by Proposition \ref{pp:nsmcubic} (1)). Thus, $k\le 2$.

Next we show that $g$ admits an $F$-lifting in $\SL(5,\C)$. Suppose $k=1$. By Lemma \ref{lem:3plift}, $g$ admits an $F$-lifting, say $A$. Since ${\rm ord}(A)={\rm ord}(g)=3$, it follows that $\Det(A)=\xi_3^a$ for some $a\in \{0,1,2\}$. Choose any $b\in \Z$ such that $5b+a\equiv 0\; ({\rm mod}\; 3)$. Then $A_g:=\xi_3^bA\in \SL(5,\C)$ is an $F$-lifting of $g$.

Suppose $k=2$. Note that $g$ admits an $F$-lifting in $\SL(5,\C)$ if and only if $g^j$ admits an $F$-lifting in $\SL(5,\C)$. Thus, if necessary, we may replacing $g$ by a suitable power $g^j$ where $3\nmid j$. Choose any $F$-lifting, say $A$, of $g$. Then, as in the proof of Lemma \ref{lem:3lift} and as in the argument above to ruling out $k=3$ (roughly speaking, if necessary, replacing the pair $(A, g)$ by a suitable power $(A^j,g^j)$,using linear change of coordinates and Proposition \ref{pp:nsmcubic} repeatedly), we may assume $A=\Diag(\xi_9,\xi_9^{-2},\xi_9^{4},\xi_9^{a_1},\xi_9^{a_2})$ where $0\le a_2\le a_1\le 8$. Then, by smoothness of $X$ and using case by case check with the help of computer algebra (e.g., Mathematica), one can conclude that $(a_1,a_2)=$ $(0,0)$, $(3,0)$, $(3,3)$, $(6,0)$, $(6,3)$, or $(6,6)$. Then $({\rm det}(A))^3=1$. Thus, for suitable $i\in\{0,1,2\}$, $\xi_3^iA$ is in $\SL(5,\C)$ and $\xi_3^iA$ is an $F$-lifting of $g$.

Finally we show uniqueness. Let $A_1, A_2\in \SL(5,\C)$ be two $F$-lifting of $g$. Then $A_1=\lambda A_2$ for some $\lambda\in \C^*$ such that $\lambda^5=1$. Since ${\rm ord}(A_1)={\rm ord}(A_2)=3^k$, it follows that $\lambda^{3^k}=1$. Thus, $\lambda=1$ and $A_1=A_2$.  This completes the proof of the proposition. \end{proof}

The following theorem will be frequently used in the later sections.

\begin{theorem}\label{thm:Flift}
Let $G\subset \Aut(X)$ be a subgroup. Then $G$ admits an $F$-lifting.
\end{theorem}

\begin{proof}

Suppose $|G|=3^kp_1^{k_1}\cdots p_n^{k_n}$, where $p_i\neq3$ are distinct prime numbers, $k\geq0$, $k_i\geq0$, for all $1\leq i\leq n$. For any $i\in \{1,...,n\}$, by Theorem \ref{liftable}, a Sylow $p_i$-subgroup $G_{p_i}$ of $G$ admits an $F$-lifting, say $H_i$. For a Sylow $3$-subgroup $G_3$ of $G$, let $$H_0=\{A_g | g\in G_3 \text{ and } A_g \text{ is the unique }F\text{-lifting of }g \text{ in }\SL(5,\C) \}.$$ Note that $H_0$ is well-defined by Proposition \ref{pp:3plift}. Let $\widetilde{G}=<H_0,H_1,\cdots,H_n><\GL(5,\mathbb{C})$. Next we show $\widetilde{G}$ is an $F$-lifting of $G$. By definition of $\widetilde{G}$, we have $A(F)=F$ for any $A\in \widetilde{G}$. Recall that $\pi: \GL(5,\C)\longrightarrow \PGL(5,\C)$ is the natural quotient map. Clearly $\pi(\widetilde{G})=G$. Let $A\in ( \Ker(\pi)\cap \widetilde{G})$. Then $A=\lambda I_5$ for some $\lambda\in \C^*$. By $A(F)=F$ and $F$ is of degree $3$,  we have $A=\xi_3^j I_5$ for some $j\in \{0,1,2\}$. Then we have $\Det(A)=\xi_3^{5j}$. On the other hand $A=B_1^{l_1}B_2^{l_2}\cdots B_m^{l_m}$, where $B_i\in (H_0\bigcup H_1\bigcup \cdots \bigcup H_n)$ and $l_i\in \Z$.
 Then $\Det(A)=\Det(B_1)^{l_1}\cdots \Det(B_m)^{l_m}$ and $\Det(A)^{p_1^{k_1}\cdots p_n^{k_n}}=1$. Thus, $(\xi_3^{5j})^{p_1^{k_1}\cdots p_n^{k_n}}=1$ and $A=I_5$. Then ${\rm Ker}(\pi)\cap \widetilde{G}=\{I_5\}$ and $\widetilde{G}$ is an $F$-lifting of $G$. This completes the proof of the theorem. \end{proof}

\begin{remark}
Unlike a smooth cubic threefold, the automorphism group of a smooth quintic threefold does not necessarily have an $F$-lifting (see \cite[Section 4]{OY19}).

\end{remark}

\medskip

\section{Abelian subgroups}\label{ss:Ab}

{\it Notation: In Sections \ref{ss:Ab}-\ref{ss:nonsol}, $X$ is a smooth cubic threefold defined by an irreducible homogeneous polynomial $F$.}

In this section, we classify abelian groups which can faithfully act on smooth cubic threefolds (Theorem \ref{thm:ab}).

 \begin{proposition}\label{pp:porder}
Let $g\in \Aut(X)$  be of primary order. Then ${\rm ord}(g)=2^a,3^b,5,$ or $11$, where $a,b>0$. In particular, if a prime number $p$ divides $|\Aut(X)|$, then $p\in \{2,3,5,11\}$.
\end{proposition} 
  
  \begin{proof}
  
  Consider the numbers $(1-3)^l-1$, for $1\leq l \leq 5$. These five numbers are $-3$, $3$, $-9$, $15$, $-33$. Thus, by \cite[Theorem~1.3]{GL13} (see also \cite[Theorem~5.1]{OY19}),   ${\rm ord}(g)$ is $2^a,3^b,5$ or $11$. \end{proof}

\medskip

\begin{theorem}\label{thm:sylow5}
If $5$ divides $|\Aut(X)|$, then a Sylow $5$-subgroup of $\Aut(X)$ is isomorphic to $C_5$.
\end{theorem}

\begin{proof}
Any finite $p$-group of order $p^k$ contains a subgroup of order $p^l$ for any $l\in\{0,1,...,k\}$. Thus, in order to prove the theorem, it suffices to show that $\Aut(X)$ contains no subgroup of order $25$. Suppose a subgroup $G\subset \Aut(X)$ is of order $25$. By Proposition \ref{pp:porder}, $G$ is not isomorphic to $C_{25}$. Then $G\cong C_5^2$. Suppose $G=\langle g_1,g_2\rangle$ where ${\rm ord}(g_1)={\rm ord}(g_2)=5$ and $g_1g_2=g_2g_1$. By Theorem \ref{thm:Flift}, $G$ admits an $F$-lifting, say $\tilde{G}$. Let $A_i\in \tilde{G}$ be $F$-lifting of $g_i$, $i=1,2$. Since $A_1A_2=A_2A_1$, under suitable linear change of coordinates, we may assume both $A_i$ are diagonal matrices. Then $A_i=\Diag(\xi_5^{a_{i1}},...,\xi_5^{a_{i5}})$, $i=1,2$. As in the proof of Proposition \ref{pp:3plift}, by smoothness of $X$ and $A_1(F)=F$, we may assume $A_1=\Diag(\xi_5,\xi_5^{-2},\xi_5^{4},\xi_5^{2},\xi_5^{a_{15}})$ and $x_1^2x_2,x_2^2x_3,x_3^2x_4\in F$. Then, replacing $A_2$ by $A_1^jA_2$ for suitable $j$, we may assume $a_{21}=0$. Then by $A_2(F)=F$ and $x_1^2x_2,x_2^2x_3,x_3^2x_4\in F$, we have $A_2=\Diag(1,1,1,1,\xi_5^{a_{25}})$. Since ${\rm ord}(A_2)=5$, it follows that $\xi_5^{a_{25}}\neq 1$. Then $A_2(x_5^2x_i)\neq x_5^2x_i$ for any $i\in \{1,...,5\}$. Thus, by $A_2(F)=F$, $x_5^2x_i\notin F$ for any $i$, a contradiction to Proposition \ref{pp:nsmcubic} (1). 
\end{proof}

\begin{theorem}[{\cite[Proposition 1.1]{Ro09}, see also \cite[Theorem 2.10]{GL11}}]\label{thm:sylow11}
If $11$ divides $|\Aut(X)|$, then $X$ is isomorphic to Klein cubic threefold (in particular, $\Aut(X)\cong \PSL(2,11)$ \cite{Ad78}).
\end{theorem}

\begin{theorem}\label{thm:ab}
Let $G\subset \Aut(X)$ be an abelian subgroup. Then $G$ is isomorphic to one of the following $25$ groups: $C_2$, $C_3$, $C_4$, $C_2^2$, $C_5$, $C_6$, $C_8$, $C_4\times C_2$, $C_9$, $C_3^2$, $C_{11}$, $C_{12}$, $C_{6}\times C_2$, $C_{15}$, $C_{16}$, $C_{18}$, $C_6\times C_3$, $C_{24}$, $C_{12}\times C_2$, $C_{9}\times C_3$, $C_3^3$, $C_{12}\times C_3$, $C_6^2$, $C_6\times C_3^2$, $C_3^4$. Moreover, up to linear change of coordinates (equivalently, up to conjugation in $\GL(5,\C)$), one of the $\tilde{H}$ in the Table \ref{tab:ab} is an $F$-lifting of $G$.
\end{theorem}

\begin{proof}
The idea of the proof is the following: firstly, by Theorem \ref{thm:Flift}, $G$ admits an $F$-lifting, say $\tilde{G}$; secondly, since $G$ and $\tilde{G}\subset \GL(5,\C)$ are abelian groups, we may assume $\tilde{G}$ consists of diagonal matrices; lastly, we use smoothness of $X$ and Proposition \ref{pp:nsmcubic} to find an explicit $\tilde{G}$ (resp. to rule out $G$) if $G$ is in the list of the $25$ groups in the theorem (resp. if $G$ is not in the list). We give the details of the proof for two groups: $C_9$, $C_4^2$, and we leave the details for other groups to the readers (note that by consideration of subgroups, we only need to rule out {\it finitely} many abelian groups, and hence we can do case by case check).

Suppose $G\cong C_9$. Then $G$ admits an $F$-lifting $\tilde{G}$. Let $A\in \tilde{G}$ be a generator. We may assume $A$ is a diagonal matrix and $A=\Diag(\xi_9,\xi_9^a,\xi_9^b,\xi_9^c,\xi_9^d)$ where $0\le a,b,c,d\le 8$. By smoothness of $X$ and Proposition \ref{pp:nsmcubic} (see proof of Proposition \ref{pp:3plift}), we may assume $x_1^2x_2,x_2^2x_3\in F$ and $a=7$, $b=4$. Then, without loss of generality,  we may assume $A=\Diag(\xi_9,\xi_9^7,\xi_9^4,\xi_9^c,\xi_9^d)$ where $0\le c \le d \le 8$. Then by considering all the 45 possibilities for the pair $(c,d)$ via direct computation, smoothness of $X$ implies $(c,d)=(0,0)$, $(0,3)$, $(0,6)$, $(3,3)$, $(3,6)$, or $(6,6)$. Thus, $\tilde{G}$ is one of  the groups $\tilde{H}\subset \GL(5,\C)$ in case No. 12 of Table \ref{tab:ab}. Then $A^3=\Diag(\xi_3,\xi_3,\xi_3,1,1)$. Thus, for any monomial $m=x_1^{i_1}...x_5^{i_5}$ of degree 3, $A(m)=m$ if and only if either $i_1=i_2=i_3=0$ or $i_4=i_5=0$. Then $F=R(x_1,x_2,x_3)+S(x_4,x_5)$ where $R$ and $S$ are homogeneous polynomial of degree 3. Since $X$ is smooth, by Jacobian test, the hypersurface defined by $R$ (resp. by $S$) in $\P^2$ (resp. $\P^1$) is a smooth cubic plane curve (resp. three distinct points). Thus, up to linear change of coordinates, $S(x_4,x_5)$ is the same as $x_4^3+x_5^3$. Since $A(F)=F$, it follows that $A(R)=R$. Then $$R=\alpha_1 x_1^2x_2+\alpha_2 x_2^2x_3+\alpha_3 x_3^2x_1,$$ where all $\alpha_i$ are nonzero complex numbers. Then we may assume all $\alpha_i=1$ and $$F= x_1^2x_2+ x_2^2x_3+ x_3^2x_1 +x_4^3+x_5^3.$$  Therefore, we have proved the following: any smooth cubic threefold admiting an order $9$ automorphism is isomorphic to the smooth cubic threefold defined by $x_1^2x_2+ x_2^2x_3+ x_3^2x_1 +x_4^3+x_5^3$. On the other hand, the Fermat cubic threefold $X_1$ admits an order $9$ automorphism. Thus, $X\cong X_1$. In fact,  let $$B=\frac{1}{\sqrt[3]{9}}
    \begin{pmatrix}
   \xi_{18}^6& \xi_{18}^6& \xi_{18}^6&0&0\\
    \xi_{18}^2& -\xi_{18}^5& \xi_{18}^8&0&0\\
    \xi_{18}^4&- \xi_{18}&- \xi_{18}^7&0&0\\
   0&0&0&{\sqrt[3]{9}}&0\\
   0&0&0&0&{\sqrt[3]{9}}
    \end{pmatrix}$$ Then $B( x_1^2x_2+ x_2^2x_3+ x_3^2x_1 +x_4^3+x_5^3)=x_1^3+x_2^3+x_3^3+x_4^3+x_5^3$. 
    
  Next suppose $G\cong C_4^2$.  Then we may assume an $F$-lifting $\tilde{G}$ of $G$ is generated by $A_1=\Diag(\xi_4,-1,1,1,1)$ and $A_2=\Diag(1,1,1,\xi_4,-1)$, and $x_1^2x_2,x_2^2x_3,x_4^2x_5\in F$ (cf. the proof of Theorem \ref{thm:sylow5}). Then $A_1(F)=A_2(F)=F$ implies $F\in (x_3)+(x_1,x_4)^2$, a contradiction to smoothness of $X$ by Proposition \ref{pp:nsmcubic} (3). Thus, $G\cong C_4^2$ is impossible. \end{proof}

\section{Sylow p-subgroups, $p=2,3$}\label{ss:Sylow23}

In this section, we classify $p$-groups in $\Aut(X)$ for $p=2,3$ (Theorems \ref{thm:sylow2}, \ref{thm:sylow3}).

Let $\mathcal{G}_2$ be the set of the following 7 groups: $C_2$, $C_4$, $C_2^2$, $C_8$, $C_4\times C_2$, $D_8$, $C_{16}$.
One can varify that a $2$-group $G$ is, as an abstract group, isomorphic to a subgroup of one of the 6 groups in Theorem \ref{thm:Main} if and only if $G\in \mathcal{G}_2$. 

\begin{theorem}\label{thm:sylow2}
Let $G\subset \Aut(X)$ be a $2$-group. Then $G\in \mathcal{G}_2$. In particular, $|G|\le 2^4$.
\end{theorem}

Before we start the proof of Theorem \ref{thm:sylow2}, we explain the main ideas of the proof (i.e., how to exclude all other 2-groups which are not in $\mathcal{G}_2$).

  We will exclude groups inductively (from smaller orders to larger orders). Our strategies to exclude groups consist of two steps:

{\bf Step one:} Let $G$ be a $2$-group of order $2^n$. If  Theorem \ref{thm:sylow2} has been proved for 2-groups of orders strictly less than $2^n$ and $G$ contains a proper subgroup which is not in $\mathcal{G}_2$, then the group $G$ is excluded. In this section and later sections, we call this method of excluding groups as {\it sub-test}. We frequently use GAP to do sub-test. The detailed GAP codes can be found on the second author\rq{}s personal website \cite{Yu}.

{\bf Step two:} If $G$ survives after sub-test and $G\notin \mathcal{G}_2$, then we just do case by case consideration to rule out $G$.

We now start to prove Theorem \ref{thm:sylow2}.

\begin{proof}[Proof of Theorem \ref{thm:sylow2}] By sub-test, it suffices to rule out the following 2-groups: $Q_8$, $C_2^3$, $C_4^2$, $C_4\rtimes C_4$ (GAP ID: $[16,4]$), $C_8\times C_2$, $C_8\rtimes C_2$ ($[16,6]$), $D_{16}$, $C_{32}$. Then by Theorem \ref{thm:ab}, we only need to rule out 4 groups: $Q_8$, $C_4\rtimes C_4$ ($[16,4]$), $C_8\rtimes C_2$ ($[16,6]$), and $D_{16}$. The ideas of our proof for these 4 groups $G$ are the same: using Table \ref{tab:ab} and character table of $G$, one can prove that none of 5-dimensional faithful linear representations of $G$ can be an $F$-lifting of $G$. We give detailed proof for the case $G=Q_8$, and leave the details for the other cases to the readers.

Suppose $Q_8\cong G\subset \Aut(X)$. By Theorem \ref{thm:Flift}, $G$ admits an $F$-lifting $\tilde{G}$. Since $\tilde{G}\subset \GL(5,\C)$, $\tilde{G}$ is a 5-dimensional faithful linear representation of $Q_8$. Consider the character table (Table \ref{tab:Q8}) of $Q_8$.

\begin{table}\caption{Character table of $Q_8$ }\label{tab:Q8}
\begin{center}

\begin{tabular}{|c|c|c|c|c|c|}

\hline 
 & $1a$ & $4a$ &$4b$ & $2a$ & $4c$ \\
\hline 
$\chi_1$& $1$ & $1$ & $1$ & $1$ & $1$  \\
\hline 
$\chi_2$& $1$ & $-1$ & $1$ & $1$ & $-1$  \\
\hline
$\chi_3$& $1$ & $1$ & $-1$ & $1$ & $-1$  \\
\hline
$\chi_4$& $1$ & $-1$ & $-1$ & $1$ & $1$  \\
\hline 
$\chi_5$& $2$ & $0$ & $0$ & $-2$ & $0$  \\
 \hline 

\end{tabular}

\end{center}

\end{table}

The group $Q_8$ has exactly 5 conjugacy classes the order of representative of which are $1$, $2$, $4$, $4$, $4$. In the table, we use $1a$, $2a$, $4a$, $4b$, $4c$ to denote these conjugacy classes. $\chi_1,...,\chi_5$ are the characters of the five irreducible representations of $Q_8$.

We denote by $\chi$ the character of $Q_8$ cooresponding to the representation $\tilde{G}$. Since the representation is faithful, it follows that a) $\chi=\chi_5+\chi_i +\chi_j+\chi_k$, where $i,j,k\in \{1,2,3,4\}$;  or b) $\chi=2 \chi_5+\chi_i $, where $i,j,k\in \{1,2,3,4\}$.

Case a): $\chi=\chi_5+\chi_i +\chi_j+\chi_k$. Then by classification of $F$-liftings of order $4$ elements in $\Aut(X)$ (see Table \ref{tab:ab} No.4), $-1$ must be an eigenvalue of any order $4$ element in $\tilde{G}$. Thus, by Table \ref{tab:Q8},  $\chi=\chi_5+\chi_2 +\chi_3+\chi_k$, $\chi=\chi_5+\chi_2 +\chi_4+\chi_k$, or $\chi=\chi_5+\chi_3 +\chi_4+\chi_k$ where $k\in \{1,2,3,4\}$. Then one of the three traces ${\rm tr}(4a)$, ${\rm tr}(4b)$, or ${\rm tr}(4c)$ (by abuse of notation,  we use ${\rm tr}(4a)$ etc. to denote the trace of a representative of a conjugacy class) must be $-1$ or $-3$, a contradiction to Table \ref{tab:ab} No. 4. Thus, case a) is impossible.

Case b): $\chi=2 \chi_5+\chi_i $. Then ${\rm tr}(2a)=-3$, a contradiction to Table \ref{tab:ab} No. 1. Thus, this case is impossible. 

Therefore, $Q_8$ is not a subgroup of $\Aut(X)$. \end{proof}

Next, similar to 2-groups, we classify 3-groups in $\Aut(X)$. Let $\mathcal{G}_3$ be the set of the following $12$ groups: $C_3$, $C_9$, $C_3^2$, $C_9\times C_3$, $C_3^2\rtimes C_3$ ([27,3]), $C_9\rtimes C_3$ ([27,4]), $C_{3}^3$, $C_3^3\rtimes C_3$ ([81,7]), $C_3\times (C_3^2\rtimes C_3)$ ([81,12]), $C_3\times (C_9\rtimes C_3)$ ([81,13]), $C_3^4$, $C_3\times (C_3^3\rtimes C_3)$ ([243, 51]).
Note that a $3$-group $G$ is isomorphic to a subgroup of one of the 6 groups in Theorem \ref{thm:Main} if and only if $G\in \mathcal{G}_3$. 

\begin{lemma}\label{lem:C9C34}
Let $G\subset \Aut(X)$ be a $3$-group. If $G$ contains either $C_9$ or $C_3^4$, then $G$ is isomorphic to a subgroup of the automorphism group of Fermat cubic threefold (in particular, $G\in \mathcal{G}_3$). 
\end{lemma}

\begin{proof}
It is well known that the automorphism group of Fermat cubic threefold $X_1$ is isomorphic to $C_3^4\rtimes S_5$ (see Section \ref{ss:6examples} for explicit description of generators of $\Aut(X_1)$). On the other hand, if $G$ contains either $C_9$ or $C_3^4$, then, by Table \ref{tab:ab} (see also the proof of Theorem \ref{thm:ab}), $X$ is isomorphic to $X_1$ and hence $G$ is isomorphic to a subgroup of $\Aut(X_1)$. \end{proof}

\begin{theorem}\label{thm:sylow3}
Let $G\subset \Aut(X)$ be a $3$-group. Then $G\in \mathcal{G}_3$. In particular, $|G|\le 3^5$.
\end{theorem}

\begin{proof}
As in the proof of Theorem \ref{thm:sylow2}, by sub-test, it sufffices to rule out the following $14$ 3-groups: $C_{27}$, $C_9^2$, $(C_9\times C_3)\rtimes C_3$ ([81,3]), $C_9\rtimes C_9$ ([81,4]), $(C_9\times C_3)\rtimes C_3$ ([81,8]),  $(C_9\times C_3)\rtimes C_3$ ([81,9]), $C_3.(C_3^2\rtimes C_3)$ ([81,10]), $C_9\times C_3^2$,  $(C_9\times C_3)\rtimes C_3$ ([81,14]), $C_3^4\rtimes C_3$ ([243,37]), $(C_3\times (C_9\rtimes C_3))\rtimes C_3$ ([243,56]), $C_3^2\times (C_3^2\rtimes C_3)$ ([243,62]), $(C_3\times (C_3^2\rtimes C_3))\rtimes C_3$ ([243,65]), $C_3^5$. By Theorem \ref{thm:ab}, now we only need to ruling out the 10 non-abelian 3-groups in the list above.  It turns out that all the 10 non-abelian 3-groups $G$ except $(C_3\times (C_3^2\rtimes C_3))\rtimes C_3$ ([243,65]) satisfy both of the following two properties: i) $G$ contains either $C_9$ or $C_3^4$, ii) $G$ is not isomorphic to any subgroup of $\Aut(X_1)$. Thus, by Lemma \ref{lem:C9C34}, we are reduced to rule out $(C_3\times (C_3^2\rtimes C_3))\rtimes C_3$ ([243,65]).  It turns out the group $(C_3\times (C_3^2\rtimes C_3))\rtimes C_3$ ([243,65]) has no 5-dimensional faithful linear representation. Thus, by Theorem \ref{thm:Flift}, it cannot be a subgroup of $\Aut(X)$ (since otherwise, its $F$-lifting would be a  5-dimensional faithful linear representation of it). \end{proof}

\begin{proposition}\label{pp:3^5}
Let $G\subset \Aut(X)$ be a subgroup of order $3^5 k$ for some positive integer $k$. Then $G$ is isomorphic to a subgroup of $\Aut(X_1)$.
\end{proposition}

\begin{proof}
By Theorem \ref{thm:sylow3}, a Sylow $3$-subgroup $G_3$ of $G$ must be isomorphic to $C_3\times (C_3^3\rtimes C_3)$. Then $C_3^4$ is a subgroup of $G$. Then by Table \ref{tab:ab} No. 33, $X\cong X_1$. Thus, $G$ is isomorphic to a subgroup of $\Aut(X_1)$. This completes the proof of the proposition. \end{proof}

\section{Solvable subgroups of order $2^a3^b5^c$}\label{ss:sol}

\begin{theorem}\label{thm:sol2000}
Let $G\subset \Aut(X)$ be a solvable group of order $2^a 3^b 5^c\le 2000$, where $a,b,c\ge 0$. Then, as an abstract group, $G$ is isomorphic to a subgroup of $\Aut(X_i)$ for some $i\in \{ 1,...,6\}$.
\end{theorem}

\begin{proof}
By Theorems \ref{thm:sylow2}, \ref{thm:sylow3}, \ref{thm:sylow5}, $a\le 4$, $b\le 5$, $c\le 1$. Moreover, we may assume at least two of $a,b,c$ are not zero. Then, by $|G|=2^a 3^b 5^c\le 2000$, $|G|$ has the following 41 possibilities: 6, 10, 12, 15, 18, 20, 24, 30, 36, 40, 45, 48, 54, 60, 72, 80, 90, 108, 120, 135, 144, 162, 180, 216, 240, 270, 324, 360, 405, 432, 486, 540, 648, 720, 810, 972, 1080, 1215, 1296, 1620, 1944. When $|G|\in \{486, 972, 1215, 1944\}$, Theorem \ref{thm:sol2000} is just a consequence of Proposition \ref{pp:3^5}. Thus, we only need to consider the remaining 37 possibilities for $|G|$. By sub-test and Theorem \ref{thm:ab}, we are reduced to ruling out the following 24 non-abelian groups: $D_{18}$, $C_3\rtimes C_8$, $D_{24}$, $C_2\times (C_3\rtimes C_4)$, $D_{30}$, $C_9\rtimes C_4$ ([36,1]), $C_2^2\rtimes C_9$ ([36,3]), $C_3^2\rtimes C_4$ ([36,7]), $C_2\times (C_3^2\rtimes C_2)$ ([36,13]), $C_3\rtimes C_{16}$, $(C_3^2\rtimes C_3)\rtimes C_2$ ([54,5]), $C_3^3\rtimes C_2$ ([54,14]), $C_{15}\rtimes C_4$ ([60, 7]), $C_3^2\times D_8$, $C_3^2\rtimes C_8$ ([72,39]), $(C_3\times A_4)\rtimes C_2$ ([72,43]), $(C_3^2\rtimes C_3)\rtimes C_4$ ([108,11]), $C_4\times (C_3^2\rtimes C_3)$ ([108,13]), $C_6^2\rtimes C_3$ ([108,22]), $C_2^2\times (C_3^2\rtimes C_3)$ ([108,30]), $C_3^2\times (C_3\rtimes C_4)$, $C_3^2\times A_4$, $C_6\times (C_3^2\rtimes C_3)$ ([162, 48]), $C_3^2\times (C_3^2\rtimes C_4)$ ([324, 161]). Then by considering character tables of these groups and by using Table \ref{tab:ab}, one can rule out these 24 non-abelian groups (see the proof of Theorems \ref{thm:ab}, \ref{thm:sylow2} and \ref{thm:sylow3}), and we leave the details to the readers. \end{proof}

In the proof of Theorem \ref{thm:sol>2000}, we will need the following known result in group theory:

\begin{theorem}[{See, for example, \cite[Chapter 4,~Theorem 5.6]{Su86}}]\label{thm:solsub} Let $G$ be a finite solvable group. We can write $$|G|=mn\;\;\;\ (m,n)=1.$$

Then, the following propositions hold.

(i) There are subgroups of order $m$.

(ii) Any two subgroups of order $m$  are conjugate.

(iii) Any subgroup whose order divides $m$ is contained in a subgroup of order $m$.
\end{theorem}

\begin{theorem}\label{thm:sol>2000}
Let $G\subset \Aut(X)$ be a solvable group of order $2^a 3^b 5^c> 2000$, where $a,b,c\ge 0$. Then, as an abstract group, $G$ is isomorphic to a subgroup of $\Aut(X_i)$ for some $i\in \{ 1,...,6\}$.
\end{theorem}

\begin{proof}
By Theorems \ref{thm:sylow2}, \ref{thm:sylow3}, \ref{thm:sylow5}, $a\le 4$, $b\le 5$, $c\le 1$. Moreover, we may assume at least two of $a,b,c$ are not zero. Then, by $|G|=2^a 3^b 5^c> 2000$, $|G|$ has the following 8 possibilities: 2160, 2430, 3240, 3888, 4860, 6480, 9720, 19440.

When $|G|\in \{2430, 3888, 4860, 9720, 19440\}$, Theorem \ref{thm:sol>2000} is just a consequence of Proposition \ref{pp:3^5}. Thus, we only need to consider the remaining 3 possibilities for $|G|$: $2160=2^4 3^3 5$, $3240=2^3 3^4 5$, $6480=2^4 3^4 5$. If $|G|\in \{2160,3240,6480\}$, then, by Theorem \ref{thm:solsub}, $G$ contains a subgroup of order $2^3 5$ or $2^4 5$, a contradiction to Theorem \ref{thm:sol2000} (note that, for any $1\le i\le 6$, $\Aut(X_i)$ contains no subgroup of order $2^3 5$ or $2^4 5$). \end{proof}

\section{Non-solvable subgroups of order $2^a3^b 5^c$}\label{ss:nonsol}

\begin{theorem}\label{thm:nsol2000}
Let $G\subset \Aut(X)$ be a non-solvable group of order $2^a 3^b 5^c\le 2000$, where $a,b,c\ge 0$. Then, as an abstract group, $G$ is isomorphic to a subgroup of $\Aut(X_i)$ for some $i\in \{ 1,...,6\}$.
\end{theorem}

\begin{proof}
By Theorems \ref{thm:sylow2}, \ref{thm:sylow3}, \ref{thm:sylow5}, $a\le 4$, $b\le 5$, $c\le 1$. A finite group with cyclic Sylow 2-subgroups is solvable (see \cite[Proposition 8.31]{OY19}). By Burnside theorem, a finite group of order $p^\alpha q^\beta$ is solvable, where $p,q$ are two distinct prime numbers and $\alpha,\beta$ are non-negative integers. Thus, we may assume $a\in \{2,3\}$, $b\in \{1,2,3,4,5\}$, $c=1$. By Proposition \ref{pp:3^5}, we may assume $b<5$. Thus,  we only need to consider the following cases for $|G|$: $60$, $120$, $180$, $360$, $540$, $1080$, $1620$. Then, by sub-test, we are reduced to ruling out the following two non-solvable groups:  $A_6$, $C_3.A_6$ ([1080,260]).  Then by considering character tables of these two groups and by using Table \ref{tab:ab}, one can rule out these 2 groups (see the proof of Theorems \ref{thm:ab}, \ref{thm:sylow2} and \ref{thm:sylow3}), and we leave the details to the readers. \end{proof}

\begin{theorem}\label{thm:nsol>2000}
Let $G\subset \Aut(X)$ be a non-solvable group of order $2^a 3^b 5^c> 2000$, where $a,b,c\ge 0$. Then, as an abstract group, $G$ is isomorphic to a subgroup of $\Aut(X_i)$ for some $i\in \{ 1,...,6\}$.
\end{theorem}

\begin{proof}
As in the proof of Theorem \ref{thm:nsol2000}, we may assume $a\in \{2,3\}$, $b\in \{1,2,3,4\}$, $c=1$. Then, by $|G|>2000$, we only need to consider the case $|G|=2^3 3^4 5=3240$. Suppose $|G|=3240$. Let $N$ be a maximal proper normal subgroup of $G$. Consider the following exact sequence $$1\longrightarrow N \longrightarrow G \longrightarrow M \longrightarrow 1.$$ By choice of $N$, $M$ is a finite simple group. By classification of finite simple groups, $M\cong C_2$, $C_3$, $C_5$,  $A_5,$ or $A_6$. 

Suppose $M\cong C_2$ (resp. $M\cong C_3$). Then $N\subset \Aut(X)$ is non-solvable and $|N|=1620$ (resp.  $|N|=1080$), a contradiction to Theorem \ref{thm:nsol2000} (in fact, for any $1\le i\le 6$, $\Aut(X_i)$ does not contain a non-solvable subgroup of order $1620$ (resp. $1080$)).

Suppose $M\cong C_5$. Then $|N|=2^3 3^4$ and $N$ is solvable. Then $G$ is solvable, a contradiction.

Suppose $M\cong A_5$. Then $|N|=2\cdot 3^3$. Since $C_5$ is a subgroup of $M$, it follows that $G$ contains a (solvable) subgroup of order $2\cdot 3^3 5=270$, a contradiction to Theorem \ref{thm:sol2000}.

Suppose $M\cong A_6$. Then $|N|=3^2$. Since $A_5$ is a subgroup of $M$, it follows that $G$ contains a (non-solvable) subgroup of order $2^2 3^3 5=540$, a contradiction to Theorem \ref{thm:nsol2000}. 

This completes the proof of the theorem. \end{proof}

\section{Proof of main Theorem}\label{ss:proofmainthm}

In this section, we prove our main Theorem (Theorem \ref{thm:Main}).

Let $G\subset \Aut(X)$ be a subgroup, where $X$ is a smooth cubic threefold. Then, by Proposition \ref{pp:porder} and Theorems \ref{thm:sylow2}, \ref{thm:sylow3}, \ref{thm:sylow5}, \ref{thm:sylow11}, it follows that $$|G|=2^{a_2}3^{a_3}5^{a_5}11^{a_{11}} \, ,$$ where $0\leq a_2\leq 4$, $ 0\leq a_3\leq 5$, $0\leq a_5\leq 1$, $0\leq a_{11}\leq 1$.

If $a_{11}$ is not zero, by Theorem \ref{thm:sylow11}, $G$ is isomorphic to a subgroup of $\Aut(X_5)$.

If $a_{11}=0$, then, by Theorems \ref{thm:sol2000}, \ref{thm:sol>2000}, \ref{thm:nsol2000}, \ref{thm:nsol>2000}, $G$ is isomorphic to a subgroup of $\Aut(X_i)$ for some $1\le i\le 6$. This completes the proof of Theorem \ref{thm:Main}.

\appendix

\section{Computer program GAP}\label{ap:GAP}
In this paper,  we extensively use the mathematical software GAP (\cite{GAP2014}). In GAP library, groups of order $\leq 2000$ (except 1024) are stored. All the information we need (structure descriptions, lists of subgroups, character tables, etc.) of these groups are included in GAP. 

A terminology used in GAP: $\SmallGroup(a,b)$:= the $b$-th group of order $a$ (here $a\leq 2000$ and $a\neq 1024$). For example, by classification, up to isomorphism, there are exactly five different groups of order $8$: $C_8,C_4\times C_2, D_8,Q_8,C_2^3$.  In GAP, these five groups are stored in a specific order. In fact,  $\SmallGroup(8,1)\cong C_8$,  $\SmallGroup(8,2)\cong C_4\times C_2$,  $\SmallGroup(8,3)\cong D_8$,  $\SmallGroup(8,4)\cong Q_8$,  $\SmallGroup(8,5)\cong C_2^3$. 

Throughout this paper, if no confuse causes, we use the following convention: Let $0< a\leq 2000$ and $a\neq 1024$. Suppose, up to isomorphism, there are $k_a$ many different finite groups of order $a$. Let $0<b\leq k_a$. Then we denote by $[a,b]$ a group isomorphic to $\SmallGroup(a,b)$. In fact, in GAP, $[a,b]$ is regarded as the \lq\lq{}ID\rq\rq{} of $\SmallGroup(a,b)$. We also call $[a,b]$ the \lq\lq{}GAP ID\rq\rq{} of groups isomorphic to $\SmallGroup(a,b)$. For example, the group $[8,3]$ is isomorphic to the dihedral group $D_8$. When the group structure is clear from the structure description of a group (e.g., $C_2^3$), we often omit its GAP ID.

\section{The table of abelian subgroups}\label{ap:tab}

 In the last column of  Table \ref{tab:ab}, blank means we do not adress uniqueness of $X$ ( so could be either unique or not unique).

 When Sylow 3-subgroup $H_3$ is not trivial, $H$ may admit several $F$-liftings (even up to conjugation in $\GL(5,\C)$). However, if $K_1$ and $K_2$ are two $F$-liftings of $H$, then $\langle K_1, \xi_3 I_5 \rangle=\langle K_2, \xi_3 I_5 \rangle\subset \GL(5,\C)$. Thus, a cubic monomial is invariant by $K_1$ if and only if it is invariant by $K_2$. Because of this reason, for simplicity, in Table \ref{tab:ab}, when $H_3$ is not a trivial group, we do not give all possible $F$-liftings of $H$ (but, at least one $F$-fifting of $H$ is given).

 Note that one can determine fixed point loci $X^H$ using ideas and results in the current paper (cf. \cite{Yu17} ). 

\begin{table}
\caption{$F$-lifting of abelian subgroups of $\Aut(X)$ }\label{tab:ab}
\begin{center}
{\footnotesize

\begin{tabular}{|c|c| p{8cm}|c|}

\hline 
No. &$H\subset \Aut(X)$  & generator(s) of an $F$-lifting $\tilde{H}$ of $H$  & Uniqueness of $X$ \\
\hline 
1& $C_2$ & $\Diag(-1,(-1)^a,1,1,1)$ $a=0,1$ & \\
\hline 
2& $C_3$ &$\Diag(1,1,1,\xi_3,\xi_3^a)$ $a=0,1,2$   & \\
\hline
3&$C_3$ &$\Diag(1,1,\xi_3,\xi_3,\xi_3^2)$   & \\
\hline
4&$C_4$ &$\Diag(\xi_4,-1,1,\xi_4^a,1)$ $a=0,2,3$   & \\
\hline 
 5&$C_2^2$ &$\Diag(-1,(-1)^a,1,1,1)$, $\Diag(1,(-1)^a,-1,1,1)$, $a=0,1$   & \\
 
\hline 
6&$C_5$ &$\Diag(\xi_5,\xi_5^3,\xi_5^4,\xi_5^2,1)$   & \\
\hline 
7& $C_2\times C_3$ &$\Diag(-1,1,-1,1,1)$, $\Diag(1,1,1,\xi_3,1)$   & \\
\hline 
 8&$C_2\times C_3$ &$\Diag(-1,1,-1,1,1)$, $\Diag(1,1,\xi_3,\xi_3,\xi_3^a)$, $a=0,1,2$    & \\
\hline 
 9&$C_2\times C_3$ &$\Diag(-1,1,1,1,1)$, $\Diag(1,1,\xi_3,\xi_3^a,\xi_3^b)$, $0\le a \le b\le 2$   & \\
\hline 
10&$C_8$ &$\Diag(\xi_8,\xi_8^6,-1,1,\xi_8^a)$, $a=0,2$   & \\
\hline 
 11&$C_4\times C_2$ &$\Diag(\xi_4,-1,1,1,1)$, $\Diag(1,1,1,-1,1)$   & \\
\hline 
12&$C_9$ &$\Diag(\xi_9,\xi_9^7,\xi_9^4,\xi_3^a,\xi_3^b)$, $0\le a\le b\le 2$   & $X\cong X_1$ \\
\hline 
 13&$C_3^2$ &$\Diag(1,\xi_3,1,\xi_3^a,\xi_3^b)$, $\Diag(1,1,\xi_3,\xi_3^c,\xi_3^d)$, $0\le a\le b\le 2$, $0\le c\le d\le 2$   & \\
\hline 
14&$C_{11}$ &$\Diag(\xi_{11},\xi_{11}^9,\xi_{11}^4,\xi_{11}^3,\xi_{11}^5)$ & $X\cong X_5$ \\
\hline 
 15&$C_4\times C_3$ &$\Diag(\xi_4,-1,1,1,1)$, $\Diag(1,1,1,\xi_3,\xi_3^{a})$ $a=0,1,2$  & \\
\hline 
16& $C_4\times C_3$ &$\Diag(\xi_4,-1,1,1,-1)$, $\Diag(1,1,1,\xi_3,\xi_3)$  & \\
\hline 
 17&$C_4\times C_3$ &$\Diag(\xi_4,-1,1,1,\xi_4^3)$, $\Diag(1,1,1,\xi_3,1)$  & \\
\hline 
18& $C_2^2\times C_3$ &$\Diag(-1,1,1,1,1)$, $\Diag(1,1,-1,1,1)$, $\Diag(1,1,\xi_3,\xi_3,\xi_3^a)$, $a=0,1,2$  & \\
\hline 
19&$C_2^2\times C_3$ &$\Diag(-1,1,1,1,(-1)^a)$, $\Diag(1,1,-1,1,(-1)^a)$, $\Diag(1,1,1,\xi_3,1)$, $a=0,1$  & \\
\hline 
20& $C_5\times C_3$ &$\Diag(\xi_5,\xi_5^3,\xi_5^4,\xi_5^2,1)$, $\Diag(1,1,1,1,\xi_3)$  & $X\cong X_6$\\
\hline 
21&$C_{16}$ &$\Diag(\xi_{16},\xi_{16}^{-2},\xi_{16}^4,-1,1)$  & $X\cong X_4$ \\
\hline 
22& $C_9\times C_2$ &$\Diag(\xi_9,\xi_9^7,\xi_9^4,\xi_3^a,\xi_3^a)$, $\Diag(1,1,1,-1,1)$, $a=0,1,2$   & $X\cong X_1$ \\
\hline 
23&$C_2\times C_3^2$ &$\Diag(-1,1,1,1,1)$, $\Diag(1,1,\xi_3,1,1)$, $\Diag(1,1,1,\xi_3,\xi_3^a)$, $a=0,1,2$  & $X\cong X_1$ if $a=1$ \\
\hline 
24&$C_2\times C_3^2$ &$\Diag(-1,1,1,1,1)$, $\Diag(1,1,\xi_3,1,\xi_3^a)$, $\Diag(1,1,1,\xi_3,\xi_3^a)$, $a=1,2$  & $X\cong X_1$ if $a=1$ \\
\hline 
25&$C_2\times C_3^2$ &$\Diag(-1,1,1,1,-1)$, $\Diag(1,1,\xi_3,1,1)$, $\Diag(1,1,1,\xi_3,\xi_3)$  & $X\cong X_1$ \\
\hline 
 26&$C_8\times C_3$ &$\Diag(\xi_8,\xi_8^6,-1,1,1)$, $\Diag(1,1,1,1,\xi_3)$   & $X\cong X_3$\\
\hline 
27& $C_4\times C_2\times C_3$ & $\Diag(\xi_4,-1,1,1,1)$,$\Diag(1,1,1,-1,1)$, $\Diag(1,1,1,\xi_3,\xi_3)$   & $X\cong X_2$ \\
\hline 
 28&$C_9\times C_3$ &$\Diag(\xi_9,\xi_9^7,\xi_9^4,1,\xi_3^a)$, $\Diag(1,1,1,\xi_3,\xi_3^b)$, $0\le a\le 2$, $b=0, 2$   & $X\cong X_1$ \\
\hline 
29&$C_3^3$ &$\Diag(1,\xi_3,1,1,\xi_3^a)$, $\Diag(1,1,\xi_3,1,\xi_3^b)$, $\Diag(1,1,1,\xi_3,\xi_3^c)$, $0\le a\le b\le c\le 2$  &  \\
\hline 
30&$C_4\times C_3^2$ &$\Diag(\xi_4,-1,1,1,1)$, $\Diag(1,1,1, \xi_3,1)$, $\Diag(1,1,1,1,\xi_3)$  &  $X\cong X_2$\\
\hline 
31&$C_2^2 \times C_3^2$ & $\Diag(-1,1,1,1,1)$,$\Diag(1,-1,1,1,1)$, $\Diag(1,\xi_3,1, \xi_3,1)$, $\Diag(1,1,1,1,\xi_3)$  & $X\cong X_1$ \\
\hline 
32&$C_2\times C_3^3$ &$\Diag(-1,1,1,1,1)$,$\Diag(1,1,\xi_3,1,1)$, $\Diag(1,1,1, \xi_3,1)$, $\Diag(1,1,1,1,\xi_3)$  & $X\cong X_1$ \\
\hline 
33&$C_3^4$ &$\Diag(1,\xi_3,1,1,1)$,$\Diag(1,1,\xi_3,1,1)$, $\Diag(1,1,1, \xi_3,1)$, $\Diag(1,1,1,1,\xi_3)$  & $X\cong X_1$ \\
\hline 

\end{tabular}
}

\end{center}

\vspace{3mm}

\end{table}

%------------------------------------------------------------------------------------------------------------------

\end{document}